\documentclass[a4paper,11pt]{amsart}
\usepackage{amsmath,amssymb,amsthm,amscd}

\title[Brown representability often fails]%
{Brown representability often fails for homotopy categories of
complexes}

\author{George Ciprian Modoi}
\address{Babe\c s--Bolyai University, Faculty of Mathematics and Computer Science \\
1, Mihail Kog\u alniceanu, 400084 Cluj--Napoca, Romania}
\email{cmodoi@math.ubbcluj.ro}

\author{Jan \v{S}\v{t}ov\'\i\v{c}ek}
\address{Charles University in Prague, Faculty of Mathematics and Physics \\
Department of Algebra \\ Sokolovska 83, 186 75 Praha 8, Czech
Republic } \email{stovicek@karlin.mff.cuni.cz}


\thanks{Second named author supported by GA\v{C}R~P201/10/P084, research project MSM~0021620839 and the DFG~Schwerpunkt SPP 1388. }
\date{\today}

\renewcommand{\iff}{if and only if }
\newcommand{\st}{such that }

\newcommand{\la}{\longrightarrow}

\newcommand{\N}{\mathbb{N}}
\newcommand{\Z}{\mathbb{Z}}

\newcommand{\Reg}{\mathfrak{R}}

\DeclareMathOperator{\Hom}{Hom}

\DeclareMathOperator{\Ext}{Ext}

\newcommand{\A}{\mathcal{A}}
\newcommand{\B}{\mathcal{B}}

\newcommand{\D}{\mathcal{D}}

\newcommand{\clL}{\mathcal{L}}
\newcommand{\clS}{\mathcal{S}}
\newcommand{\T}{\mathcal{T}}
\newcommand{\U}{\mathcal{U}}
\newcommand{\X}{\mathcal{X}}

\newcommand{\ModR}{\hbox{\rm Mod-}R}

\newcommand{\Ab}{\mathrm{Ab}}

\newcommand{\Add}{\mathrm{Add}}
\newcommand{\Prod}{\mathrm{Prod}}

\newcommand{\Htp}[1]{\mathbf{K}({#1})}

\newcommand{\Loc}{\mathrm{Loc}\,}

\theoremstyle{plain}
\newtheorem{thm}{Theorem}
\newtheorem{lem}[thm]{Lemma}
\newtheorem{prop}[thm]{Proposition}

\theoremstyle{definition}
\newtheorem{defn}[thm]{Definition}

\theoremstyle{remark}
\newtheorem{rem}[thm]{Remark}
\newtheorem{expl}[thm]{Example}

\begin{document}

%\maketitle

\begin{abstract}
We show that for the homotopy category $\Htp\Ab$ of complexes of abelian groups, both Brown representability and Brown representability for the dual fail. We also provide an example of a localizing subcategory of $\Htp\Ab$ for which the inclusion into $\Htp\Ab$ does not have a right adjoint.
\end{abstract}

\maketitle

% ------------------------------------------------------------------------------
\section*{Introduction}

Inspired by a result of Casacuberta and Neeman~\cite{CN09}, we provide in this note a class of naturally occurring triangulated categories with coproducts and products, for which both Brown representability and Brown representability for the dual (in the sense of~\cite[Definition 8.2.1]{Nee01}) fail. In particular, the homotopy category of complexes of abelian groups belongs to this class.

\smallskip

Let us shortly explain the motivation. Given a triangulated category $\T$ with coproducts and an object $Y \in \T$, the contravariant functor
$$ \T(-,Y): \T \la \Ab $$
is cohomological and sends coproducts in $\T$ to products in $\Ab$. By saying that $\T$ \emph{satisfies Brown representability}, we mean that the converse holds for $\T$. That is, each cohomological functor $\T \to \Ab$ which sends coproducts in $\T$ to products is naturally equivalent to $\T(-,Y)$ for some $Y \in \T$.

Dually, given a category $\T$ with products, we may be interested in whether covariant $\Hom$--functors $\T(Y,-)$ are characterized by the property that they are homological and send products in $\T$ to products of abelian groups. In such a case we say that $\T$ \emph{satisfies Brown representability for the dual}.

\smallskip

As noted in~\cite{CN09}, there have recently been several results confirming Brown representability for various classes of triangulated categories. A lot of attention has been given to \emph{well generated} triangulated categories; we refer to~\cite[p. 274]{Nee01} or~\cite[\S6]{Kr09} for a precise definition. This is a wide class of triangulated categories and we have the following implications (cf.~\cite[Chapter 8]{Nee01}):

\smallskip
\begin{center}
$\T$ is a well generated triangulated category. \nopagebreak[4]\\
$\Downarrow$ \nopagebreak[4]\\
$\T$ satisfies Brown representability. \nopagebreak[4]\\
$\Downarrow$ \nopagebreak[4]\\
A triangulated functor $G: \T \to \U$ which sends coproducts in $\T$ to coproducts in $\U$ has a right adjoint.
\end{center}
\smallskip

However, it turned out that many naturally occurring algebraic triangulated categories are not well generated.
In particular, the homotopy category $\Htp\ModR$ is well generated \iff $R$ is a right pure semisimple ring,
\cite[Proposition 2.6]{St08}. Here we show that for this type of categories, the second link in the implication
chain above also fails for a very formal reason. Postponing the explanation of the terminology to the next section,
we can state:

\begin{thm}\label{thm:fail-brt}
Let $\T$ be a locally well generated triangulated category. Then $\T$ satisfies Brown representability \iff $\T$
is well generated.
In particular, if $R$ is a ring which is not right pure semisimple, for instance $R=\Z$, then $\Htp\ModR$
does not satisfy Brown representability.
\end{thm}

\smallskip

When considering the Brown representability for the dual, the situation is more delicate. First, it is still an open problem whether every well generated triangulated category has this property. However, \cite[Theorem 8.4.4]{Nee01} dualizes smoothly and we have the implication:

\smallskip
\begin{center}
$\T$ satisfies Brown representability for the dual. \nopagebreak[4]\\
$\Downarrow$ \nopagebreak[4]\\
A triangulated functor $G: \T \to \U$ which sends products in $\T$ to products in $\U$ has a left adjoint.
\end{center}
\smallskip

Unfortunately, we do not have a theorem characterizing the categories which satisfy Brown representability for the dual and covering such a wide class of categories as Theorem~\ref{thm:fail-brt}. Nevertheless, we give a necessary condition for the representability showing that many homotopy categories of complexes, including $\Htp\Ab$, cannot satisfy Brown representability for the dual.

Namely, let $\B$ be an additive category with products. For a subcategory $\U \subseteq \B$, we denote by $\Prod(\U)$ the closure of $\U$ in $\B$ under products and direct summands. If $\U = \{X\}$ is a singleton, we write just $\Prod(X)$. Further, we say that $\B$ has a \emph{product generator} if there exists $X \in \B$ \st $\B = \Prod(X)$. We then obtain the following result:

\begin{thm}\label{thm:fail-brtop}
Let $\B$ be an additive category with products. If $\Htp\B$ satisfies Brown representability for the dual, then $\B$ has a product generator. In particular $\Htp\Ab$ does not satisfy Brown representability for the dual.
\end{thm}

\smallskip

Finally, we touch the most delicate question, when a triangulated functor $G: \T \to \U$ has an adjoint. As a matter of the fact, the existence of adjoint functors can sometimes be proved even if Brown representability fails. This was done by Neeman~\cite{Nee06}, and other attempts have followed in~\cite{SaSt10,K10,BEIJR}.  Using recent results from~\cite{EGPT09,SaTr10,BS10} (and postponing the explanation of the terminology again), we give, however, a statement showing that the existence of adjoints does not come for free either:

\begin{thm}\label{thm:fail-adj}
Let $R$ be a countable ring and let $\D$ be the class of all right flat Mittag--Leffler $R$--modules. Then $\Htp\D$ is always closed under coproducts in $\Htp\ModR$, but the inclusion functor $\Htp\D \to \Htp\ModR$ has a right adjoint \iff $R$ is a right perfect ring. In particular, a right adjoint does not exist for $R = \Z$.
\end{thm}

\subsection*{Acknowledgments}
The first named author would like to thank Simion Breaz for indicating him the argument showing that the category of abelian groups does not have a product generator, which is used in the proof of Theorem~\ref{thm:fail-brtop}.

% ------------------------------------------------------------------------------
\section{Representability for locally well generated triangulated categories}
\label{sec:brown_rep}

Let $\T$ be a triangulated category and denote by $\Reg$ the (proper) class of all infinite regular cardinal numbers. In what follows we often need an increasing chain
\[
\clS_{\aleph_0} \subseteq \clS_{\aleph_1} \subseteq \clS_{\aleph_2} \subseteq \dots \subseteq \clS_{\kappa} \subseteq \dots
\]
of skeletally small triangulated subcategories of $\T$ indexed by $\Reg$ \st the union $\bigcup_{\kappa \in \Reg} \clS_\kappa$ is the whole of $\T$. 

% Note: Having a well ordered universe implies the Axiom of Choice for classes!
\begin{rem}\label{rem:set_theory}
Strictly speaking, it is not clear how to obtain such a chain in general using the axioms of ZFC alone. But there are two workarounds. First, if we work with a more concrete triangulated category, it may be possible to construct the chain directly. For example, if $\T = \Htp\ModR$, then $\clS_\kappa$ can be defined as
the subcategory of all complexes formed by $\kappa$--presented modules. Second, if we insist on general $\T$, we can adopt some suitable axiomatization of set theory which allows us to well-order the universe of all sets (eg.\ the von~Neumann--Bernays--G\"odel set theory). Then we can easily construct the chain using the induced well-ordering of objects of $\T$. The same applies to the proof of Proposition~\ref{prop:well-filt} below, where we strictly speaking use the Axiom of Choice for proper classes.
\end{rem}

Given an arbitrary full subcategory $\clS\subseteq\T$, we denote
\begin{align*}
{^\perp\clS} &= \{X\in\T\mid\T(X,S)=0\hbox{ for all }S\in\clS\}  \\
\clS^\perp   &= \{X\in\T\mid\T(S,X)=0\hbox{ for all }S\in\clS\},
\end{align*}
Note that if $\clS$ is a triangulated subcategory of $\T$, so are $^\perp\clS$ and $\clS^\perp$. Now we can formulate a simple but important obstruction to Brown representability.

\begin{prop}\label{prop:well-filt}
Let $\T$ be a triangulated category with coproducts. Suppose that $\T$ possesses an increasing chain $(\clS_\kappa \mid \kappa \in \Reg)$ of skeletally small triangulated subcategories \st $\T = \bigcup_{\kappa\in\Reg}\clS_\kappa$ and $\clS_\kappa^\perp \neq 0$ for all $\kappa\in\Reg$. Then $\T$ does not satisfy Brown representability.

Dually, suppose $\T$ is triangulated, has products and has an increasing chain $(\clS_\kappa \mid \kappa \in \Reg)$ of skeletally small triangulated subcategories \st $\T = \bigcup_{\kappa\in\Reg}\clS_\kappa$ and ${^\perp\clS_\kappa} \neq 0$ for all $\kappa\in\Reg$. Then $\T$ does not satisfy Brown representability for the dual.
\end{prop}

\begin{proof}
We prove only the first part, the other is dual. Choose for each $\kappa\in\Reg$ an object $0\neq Y_\kappa\in{\clS_\kappa^\perp}$. We consider the functor
\[ F=\prod_{\kappa\in\Reg}\T(-,Y_\kappa): \T \la \Ab. \]
Note that $F$ is a well-defined functor. Indeed, recall that any $X\in\T$ is contained in $\clS_\kappa$ for some $\kappa \in \Reg$, so $\T(X,Y_\lambda) = 0$ for all $\lambda \ge \kappa$ and the product defining $FX$ is essentially set-indexed. Moreover, $F$ is homological and sends coproducts to products.

Now we are essentially done, since if $F$ were represented by some object in $\T$, it would have to be the product of $(Y_\kappa \mid \kappa \in \Reg)$ in $\T$, which cannot exist. To give a formal argument, assume for the moment that there is some $Y \in \T$ and a natural equivalence
\[ \eta: \T(-,Y)\la F. \]
For each $\kappa \in \Reg$ we then have an idempotent natural transformation $\epsilon_\kappa: F \to F$ given as the composition
\[ F\la\T(-,Y_\kappa) \la F \]
which, by the Yoneda lemma, induces an idempotent morphism $e_\kappa: Y \to Y$ in $\T$. Since $(\epsilon_\kappa \mid \kappa \in \Reg)$ is a proper class of pairwise orthogonal non-zero idempotent endotransformations of $F$, the collection $(e_\kappa \mid \kappa \in \Reg)$ would have to be a proper class of endomorphisms of $Y$ with the same properties. This is absurd since $\T(Y,Y)$ is a set.
\end{proof}

Let us next explain the terminology necessary for Theorem~\ref{thm:fail-brt}.

\begin{defn}\label{def:loc_well_gen}
Let $\T$ be a triangulated category with coproducts. A full triangulated subcategory $\clL$ of $\T$ is called \emph{localizing} if it is also closed under taking coproducts in $\T$. Given a class of objects $\clS \subseteq \T$, we denote by $\Loc\clS$ the smallest localizing subcategory of $\T$ containing $\clS$.

The category $\T$ is called \emph{locally well generated} (in the sense of~\cite[Definition 3.1]{St08}) if for any set $\clS$ (not a proper class!) of objects of $\T$, $\Loc\clS$ is well generated.
\end{defn}

Now we are ready to give a proof of Theorem \ref{thm:fail-brt}. For a more concrete construction of a non-representable functor $\Htp\Ab \to \Ab$, see Example~\ref{expl:K(Ab)} below.

\begin{proof}[Proof of Theorem \ref{thm:fail-brt}] 
If $\T$ is well generated, or equivalently $\T = \Loc\clS$ for some set $\clS$, then Brown representability holds by~\cite[Proposition 8.4.2]{Nee01}. Let us, therefore, assume that $\T$ is \emph{not} well generated.

As discussed above, we have an increasing chain $(\clS_\kappa \mid \kappa \in \Reg)$ of skeletally small triangulated subcategories of $\T$ \st $\T = \bigcup_{\kappa\in\Reg}\clS_\kappa$. Let us put $\clL_\kappa = \Loc\clS_\kappa$; by definition each $\clL_\kappa$ is well generated and our assumption ensures $\clL_\kappa \subsetneqq \T$. It follows from~\cite[9.1.13 and 9.1.19]{Nee01} that each $X \in \T$ admits a triangle
\[
\Gamma_\kappa X \la X \la L_\kappa X \la \Gamma_\kappa X[1]
\eqno{(*)}
\]
with $\Gamma_\kappa X \in \clL_\kappa$ and $L_\kappa X \in \clS_\kappa^\perp$.
%We recall that
%
%\[ \clS_\kappa^\perp = \{Y \in \T \mid \T(S,Y) = 0 \textrm{ for each } S \in \clS_\kappa \}, \]
%
%The important consequence for us is that if $\clL_\kappa \subsetneqq \T$, then $\clS_\kappa^\perp \ne 0$.
So given arbitrary $X \in \T \setminus \clL_\kappa$ it follows $0 \ne L_\kappa X \in \clS_\kappa^\perp$. Now we just apply  Proposition \ref{prop:well-filt}.

Finally, the second part of the theorem follows from~\cite[2.6 and 3.5]{St08}: If $R$ is a ring which is
not right pure semisimple, $\Htp\ModR$ is locally well generated but not well generated.
\end{proof}

% ------------------------------------------------------------------------------
\section{Representability for the dual}
\label{sec:dual}

In this section we discuss Brown representability for the dual and prove Theorem \ref{thm:fail-brtop}. Based on \cite[Definition 2.24]{M07}, we introduce the following concept:

\begin{defn}\label{def:augmented}
Let $\B$ be an additive category and $\X$ a full subcategory. Given $M \in \B$, by an \emph{augmented proper right $\X$--resolution} we understand a complex of the form
\[ X_M: \qquad \dots \la 0 \la 0 \la M \la X^0 \la X^1 \la X^2 \la \cdots, \]
\st $X^i \in \X$ for all $i \ge 0$ and $\Hom_{\Htp\B}(X_M, X'[n]) = 0$ for all $X' \in \X$ and $n \in \Z$.
\end{defn}

A favorable fact is that such resolutions often do exist.

\begin{lem}\label{lem:augmented}
Let $\B$ be an additive category with products and splitting idempotents, let $X \in \B$ and put $\X = \Prod(X)$. Then any $M \in \B$ admits an augmented proper right $\X$--resolution $X_M \in \Htp\B$. Moreover, $X_M = 0$ in $\Htp\B$ \iff $M \in \X$.
\end{lem}

\begin{rem}\label{rem:split_idemp}
The lemma is true also without $\B$ having splitting idempotents, but we keep the assumption for the sake of simplicity.
\end{rem}

\begin{proof}
We will construct the terms $X^i$ of an augmented resolution
\[ X_M: \qquad \dots \la 0 \la 0 \la M \overset{d^{-1}}\la X^0 \overset{d^0}\la X^1 \overset{d^1}\la X^2 \overset{d^2}\la \cdots \]
by induction on $i$. We put $X^0 = X^{\Hom_\B(M,X)}$ and take for $d^{-1}$ the obvious morphism. Having constructed $X^i$ for $i \ge 0$, we set
\[ \mathcal{Z}_i = \{ f \in \Hom_\B(X^i,X) \mid f \circ d^{i-1} = 0 \}, \]
Then we can take $X^{i+1} = X^{\mathcal{Z}_i}$ and construct $d^i: X^i \to X^{i+1}$ in the obvious way.

For the second part, assume that $X_M = 0$ in $\Htp\B$, so it is a contractible complex. In particular, $d^{-1}: M \to X^0$ splits, so $M \in \Prod(X) = \X$. The other implication is easy.
\end{proof}

Now we show a consequence of non-existence of a product generator for~$\B$, which is important in connection with Proposition~\ref{prop:well-filt}.

\begin{prop}\label{prop:prod_gen}
Let $\B$ be an additive category with products and splitting idempotents. If $\B$ does not have a product generator, then $^\perp\clS\neq 0$ in $\Htp\B$ for
every set (not a proper class!) $\clS \subseteq \Htp\B$.
\end{prop}

\begin{proof}
Suppose $\B$ has no product generator and $\clS \subseteq \Htp\B$ is a set of complexes. Let $\U \subseteq \B$ be the set of all objects occurring in the components of complexes in $\clS$, and let $\X = \Prod(\U)$. Then clearly $\clS\subseteq\Htp\X$, so it suffices to show that $^\perp\Htp\X \ne 0$ in $\Htp\B$.

To this end, we have $\X \subsetneqq \B$ since $\B$ has no product generator. Thus, we can take an object $M \in \B \setminus \X$ and construct, using Lemma~\ref{lem:augmented}, an augmented proper right $\X$--resolution $X_M$ of $M$ \st $X_M \ne 0$ in $\Htp\B$. We would like to see that $X_M \in {^\perp\Htp\X}$, but this has been proved by Murfet in~\cite[Proposition 2.27]{M07} (using crucially the fact that $X_M$ is a complex which is bounded below).
\end{proof}

Now we are ready to prove Theorem~\ref{thm:fail-brtop}:

\begin{proof}[Proof of Theorem~\ref{thm:fail-brtop}]
First of all, we may without loss of generality assume that $\B$ has splitting idempotents. If not, we replace $\B$ by its idempotent completion $\tilde\B$ (see e.g.~\cite[\S1]{BaSl01}). Since $\Htp\B$ has splitting idempotents by~\cite[Proposition 1.6.8 and Remark 1.6.9]{Nee01}, it follows that the inclusion $\Htp\B \subseteq \Htp{\tilde\B}$ is a triangle equivalence.

Next we suppose that $\B$ has no product generator and prove the existence of a non-representable homological product-preserving functor $F: \Htp\B \to \Ab$. Namely, we choose an increasing chain
\[
\clS_{\aleph_0} \subseteq \clS_{\aleph_1} \subseteq \clS_{\aleph_2} \subseteq \dots \subseteq \clS_{\kappa} \subseteq \dots
\]
of skeletally small triangulated subcategories of $\Htp\B$ indexed by $\Reg$ \st the union $\bigcup_{\kappa \in \Reg} \clS_\kappa$ is the whole of $\Htp\B$ (cf.\ Remark~\ref{rem:set_theory}). Then, however, Proposition~\ref{prop:prod_gen} ensures that $^\perp\clS_\kappa \ne 0$ in $\Htp\B$ for each $\kappa\in\Reg$, and so we are in the situation of Proposition~\ref{prop:well-filt}, which asserts the existence of such a functor.

Finally, we must prove that $\Ab$ has no product generator. For this purpose, let us fix a prime number $p \in \N$. Using the notation from~\cite[\S XI.65]{Fu73}, we define inductively for every abelian group $G$ and every ordinal $\sigma$:
\[
p^{\sigma}G =
\begin{cases}
G, & {\rm if\ } \sigma = 0\\
p (p^{\sigma-1}G), & {\rm if\ } \sigma\ {\rm is\ non\ limit.}\\
\bigcap\limits_{\rho<\sigma}p^{\rho}G, & {\rm if\ } \sigma\ {\rm is\ limit.}
\end{cases}
\]
The \emph{length} $l(G)$ of the group $G$ is then by definition the minimum ordinal $\lambda$ such that $p^{\lambda+1}G=p^{\lambda}G$. Note that for any family $(G_i \mid i \in I)$ of abelian groups, we have the formula
\[ l\big(\prod G_i\big) = \sup\big\{ l(G_i) \mid i \in I \big\}. \]
Thus, to prove that $\Ab$ has no product generator, it suffices to construct abelian groups of arbitrary length. However, such families of groups are known. For instance Walker's groups $P_\beta$~\cite{W74} (whose construction can also be found in~\cite[Construction C.2.1]{Nee01}) or generalized Pr\"ufer groups~\cite[pp. 85--86]{Fu73}.
\end{proof}

\begin{rem} \label{rem:no_product_generator}
The non-existence of a product generator for $\ModR$ seems to be a much more widespread phenomenon. If $X \in \ModR$ is a product generator, then $\Ext^1_R(M,X) = 0$ implies that $M$ is projective for each $M \in \ModR$. That is, $X$ is a \emph{projective test module} in the sense of~\cite[p. 408]{EM02}. If $R$ is not right perfect it is, however, consistent with ZFC~+~GCH that there are no projective test modules and in particular no product generators. We are grateful to the anonymous referee for making us aware of this fact.
\end{rem}

We conclude the section with more concrete examples of non-representable (co)homological functors $\Htp\Ab \to \Ab$.

\begin{expl}\label{expl:K(Ab)}
Let us for each $\kappa \in \Reg$ denote by $\A_\kappa$ the full subcategory of $\Ab$ formed by all groups of cardinality smaller than $\kappa$, and put $\T = \Htp\Ab$.

If we take for a given $\kappa$ a group $P_\kappa$ of length $\kappa+1$ (e.g.\ Walker's group $P_\kappa$ from~\cite{W74}), then clearly $P_\kappa \not\in \Prod(\A_\kappa)$, since the length of any group from $\Prod\A_\kappa$ is at most $\kappa$. Thus, recalling the arguments above, we see that the augmented proper right $\Prod(\A_\kappa)$--resolution of $P_\kappa$, which we denote by $Y_\kappa$, is nonzero in $\Htp\Ab$ and belongs to $^\perp\Htp{\Prod(\A_\kappa)}$. In particular, the functor
\[ F=\prod_{\kappa\in\Reg}\T(Y_\kappa,-): \T \la \Ab, \]
is a well-defined homological functor which sends products in $\T$ to products of abelian groups, but it is not representable by an object of $\T$.

Let us also explicitly construct a contravariant non-representable functor. In fact, we will use the formally dual statement to Theorem~\ref{thm:fail-brtop} and its proof for this rather than Theorem~\ref{thm:fail-brt}. The key point is~\cite[Theorem 3.1]{Ch60} by Chase, which implies that for any uncountable $\kappa \in \Reg$, we have $\Z^\kappa \not\in \Add\A_\kappa$. Here, $\Add\A_\kappa$ denotes as usual the closure of $\A_\kappa$ in $\Ab$ under taking direct sums and summands. Therefore, denoting by $Y'_\kappa$ the augmented proper left $\Add(\A_\kappa)$--resolution of $\Z^\kappa$ (with the obvious meaning), we can infer exactly as before that the functor
\[ F'=\prod_{\kappa\in\Reg}\T(-,Y'_\kappa): \T \la \Ab, \]
is a well-defined cohomological functor which sends coproducts in $\T$ to products of abelian groups, but it is not representable by an object of $\T$.
\end{expl}

% ------------------------------------------------------------------------------
\section{The non-existence of right adjoint}

The final section is focused on Theorem~\ref{thm:fail-adj}, which is in fact an easy consequence of recent results from~\cite{EGPT09,SaTr10,BS10}. Let us explain the terminology. Given a ring $R$, a right $R$--module $M$ is called \emph{Mittag--Leffler} if the canonical map of groups
\[ M \otimes_R \left(\prod_{i\in I}N_i\right) \la \prod_{i\in I}(M\otimes_RN_i) \]
is injective for each family of left $R$-modules $(N_i\mid i\in I)$. This concept comes from~\cite{RG71}.

Let $\D$ be the class of all flat Mittag--Leffler $R$--modules. There are several characterizations of modules in $\D$ already in work of Raynaud and Gruson~\cite{RG71}, but the latest one is due to Herbera and Trlifaj, \cite[Theorem 2.9]{HeTr10}: Flat Mittag-Leffler modules coincide with so called $\aleph_1$-projective modules. For $R=\Z$, this simply means that $G \in \D$ \iff each countable subgroup of $G$ is free, which is a special case of~\cite[Proposition 7]{AF89} proved by Azumaya and Facchini. Let us now prove the last theorem.

\begin{proof}[Proof of Theorem~\ref{thm:fail-adj}]
It is rather easy to see that $\D$ is closed under direct sums and contains all projective modules.

Assume first that $R$ is right perfect. Then $\D$ coincides with the class of projective modules. In particular, $\Htp\D$ is a well-generated triangulated category by~\cite[Theorem 1.1]{Nee08}, so the inclusion $\Htp\D \to \Htp\ModR$ has a right adjoint by~\cite[Theorem 8.4.4]{Nee01}.

On the other hand, assume that $\Htp\D \to \Htp\ModR$ has a right adjoint. Given any $G \in \ModR$ and considering it as a stalk complex in degree 0, we have the counit of adjunction $\varepsilon_G: D \to G$. Let us take the $R$--module homomorphism $f = \varepsilon_G^0: D^0 \to G$ in degree 0. Clearly $D^0 \in \D$ and it is easy to see that any $R$--module homomorphism $f': D' \to G$ with $D' \in \D$ factors through $f$. That is, $\D$ is what is usually called a precovering class in $\ModR$. However, according to \cite[Theorem 6]{BS10}, this implies for a countable ring $R$ that it is right perfect. 
\end{proof}

% ------------------------------------------------------------------------------
\bibliographystyle{abbrv}
\bibliography{brown_rep_bib}

\begin{thebibliography}{10}

\bibitem{AF89}
G.~Azumaya and A.~Facchini.
\newblock Rings of pure global dimension zero and {M}ittag-{L}effler modules.
\newblock {\em J. Pure Appl. Algebra}, 62(2):109--122, 1989.

\bibitem{BaSl01}
P.~Balmer and M.~Schlichting.
\newblock Idempotent completion of triangulated categories.
\newblock {\em J. Algebra}, 236(2):819--834, 2001.

\bibitem{BS10}
S.~Bazzoni and J.~{\v{S}}{\v{t}}ov{\'{\i}}{\v{c}}ek.
\newblock Flat {M}ittag--{L}effler modules over countable rings.
\newblock To appear in Proc. Amer. Math. Soc., arXiv:1007.4977v2, 2010.

\bibitem{BEIJR}
D.~Bravo, E.~Enochs, A.~Iacob, O.~Jenda, and J.~Rada.
\newblock Cotorsion pairs in {C}({R}-{M}od).
\newblock To appear in Rocky Mountain J. Math., 2010.

\bibitem{CN09}
C.~Casacuberta and A.~Neeman.
\newblock Brown representability does not come for free.
\newblock {\em Math. Res. Lett.}, 16(1):1--5, 2009.

\bibitem{Ch60}
S.~U. Chase.
\newblock Direct products of modules.
\newblock {\em Trans. Amer. Math. Soc.}, 97:457--473, 1960.

\bibitem{EM02}
P.~C. Eklof and A.~H. Mekler.
\newblock {\em Almost free modules}, volume~65 of {\em North-Holland
  Mathematical Library}.
\newblock North-Holland Publishing Co., Amsterdam, revised edition, 2002.
\newblock Set-theoretic methods.

\bibitem{EGPT09}
S.~Estrada, P.~Guil~Asensio, M.~Prest, and J.~Trlifaj.
\newblock Model category structures arising from {D}rinfeld vector bundles.
\newblock Preprint, arXiv:0906.5213v1, 2009.

\bibitem{Fu73}
L.~Fuchs.
\newblock {\em Infinite abelian groups. {V}ol. {II}}.
\newblock Academic Press, New York, 1973.
\newblock Pure and Applied Mathematics. Vol. 36-II.

\bibitem{HeTr10}
D.~Herbera and J.~Trlifaj.
\newblock Almost free modules and {M}ittag--{L}effler conditions.
\newblock Preprint, arXiv:0910.4277v1, 2009.

\bibitem{K10}
H.~Krause.
\newblock Approximations and adjoints in homotopy categories.
\newblock To appear in Math. Ann., arXiv:1005.0209v2, 2010.

\bibitem{Kr09}
H.~Krause.
\newblock Localization theory for triangulated categories.
\newblock In {\em Triangulated categories}, volume 375 of {\em London Math.
  Soc. Lecture Note Ser.}, pages 161--235. Cambridge Univ. Press, Cambridge,
  2010.

\bibitem{M07}
D.~Murfet.
\newblock {\em The Mock Homotopy Category of Projectives and {G}rothendieck
  Duality}.
\newblock PhD thesis, Australian National University, 2007.
\newblock Available at http://www.therisingsea.org/thesis.pdf.

\bibitem{Nee01}
A.~Neeman.
\newblock {\em Triangulated categories}, volume 148 of {\em Annals of
  Mathematics Studies}.
\newblock Princeton University Press, Princeton, NJ, 2001.

\bibitem{Nee08}
A.~Neeman.
\newblock The homotopy category of flat modules, and {G}rothendieck duality.
\newblock {\em Invent. Math.}, 174(2):255--308, 2008.

\bibitem{Nee06}
A.~Neeman.
\newblock Some adjoints in homotopy categories.
\newblock {\em Ann. of Math. (2)}, 171(3):2143--2155, 2010.

\bibitem{RG71}
M.~Raynaud and L.~Gruson.
\newblock Crit\`eres de platitude et de projectivit\'e. {T}echniques de
  ``platification'' d'un module.
\newblock {\em Invent. Math.}, 13:1--89, 1971.

\bibitem{SaSt10}
M.~Saor{\'{\i}}n and J.~{\v{S}}{\v{t}}ov{\'{\i}}{\v{c}}ek.
\newblock On exact categories and applications to triangulated adjoints and
  model structures.
\newblock {\em Adv. Math.}, 228(2):968--1007, 2011.

\bibitem{SaTr10}
J.~{\v S}aroch and J.~Trlifaj.
\newblock Kaplansky classes, finite character, and $\aleph_1$-projectivity.
\newblock To appear in Forum Math., doi:10.1515/FORM.2011.101.

\bibitem{St08}
J.~{\v{S}}{\v{t}}ov{\'{\i}}{\v{c}}ek.
\newblock Locally well generated homotopy categories of complexes.
\newblock {\em Doc. Math.}, 15:507--525, 2010.

\bibitem{W74}
E.~A. Walker.
\newblock The groups {$P_{\beta }$}.
\newblock In {\em Symposia {M}athematica, {V}ol. {XIII} ({C}onvegno di {G}ruppi
  {A}beliani, {INDAM}, {R}ome, 1974)}, pages 245--255. Academic Press, London,
  1974.

\end{thebibliography}

\end{document}